\documentclass[11pt]{amsart}
\usepackage{amssymb}
\usepackage{mathtools}
\usepackage{epsfig}
\usepackage{amscd}
\usepackage{amsmath}
\usepackage{graphicx}
\usepackage{wasysym}
\usepackage{enumerate}
\usepackage{ifpdf}
\usepackage{amsfonts}
\usepackage{amsthm}
\usepackage[hyperfootnotes=false]{hyperref}
\hypersetup{hidelinks}
\usepackage{setspace}
\usepackage[usenames]{xcolor}
\usepackage{tikz}
\usepackage[normalem]{ulem}
\usepackage{ulem}
 \usepackage{cancel}

\theoremstyle{plain}
\newtheorem{theorem}{Theorem}
\newtheorem{lemma}{Lemma}
\newtheorem{corollary}{Corollary}
\theoremstyle{definition}
\newtheorem{example}{Example}
\newtheorem{definition}{Definition}

\theoremstyle{remark}

 \numberwithin{equation}{section}

\DeclareMathOperator{\Zb}{\mathbb{Z}}

\newcommand{\abs}[1]{\left|#1\right|}

\newcommand{\norm}[1]{\left\|#1\right\|}

\begin{document}
\sloppy

\title[Unbounded Kobayashi hyperbolic domains in $\mathbb C^n$]{Unbounded Kobayashi hyperbolic domains in $\mathbb C^n$}

\author[H. Gaussier]{Herv\'e Gaussier
}
\address{H. Gaussier: Univ. Grenoble Alpes, CNRS, IF, F-38000 Grenoble, France}
\email{herve.gaussier@univ-grenoble-alpes.fr}

\author[N. Shcherbina]{Nikolay Shcherbina
}
\address{N. Shcherbina: Department of Mathematics, University of Wuppertal, 42119 Wuppertal, Germany}
\email{shcherbina@math.uni-wuppertal.de}

\date{\today}
\subjclass[2010]{32Q45, 32U05.}
\keywords{Kobayashi hyperbolicity, plurisubharmonic functions.}


\maketitle

\begin{abstract}
We first give a sufficient condition, issued from pluripotential theory, for an unbounded domain in the complex Euclidean space $\mathbb C^n$ to be Kobayashi hyperbolic. Then, we construct an example of a rigid pseudoconvex domain in $\mathbb C^3$ that is Kobayashi hyperbolic and has a nonempty core. In particular, this domain is not biholomorphic to a bounded domain in $\mathbb C^3$ and the mentioned above sufficient condition for Kobayashi hyperbolicity is not necessary.
\end{abstract} 

\section*{Introduction}
According to the Riemann mapping theorem, every simply-connected domain in $\mathbb C$, different from $\mathbb C$, is biholomorphically equivalent to the unit disk $\Delta_1(0):=\{\lambda \in \mathbb C:\ |\lambda| < 1\}$. It is well known that this result has no direct generalization to higher dimension, since for instance every domain in $\mathbb C^n$ containing a nonconstant entire curve cannot be biholomorphic to a bounded domain in $\mathbb C^n$. There are different tools to distinguish domains, among which invariant metrics (under the action of biholomorphisms) play an important role. We recall that if $M$ is a complex manifold, $\Delta_r(0):=\{\lambda \in \mathbb C : |\lambda| < r\}$ for every $r > 0$ and $\mathcal H(\Delta_r(0),M)$ denotes the set of holomorphic maps from $\Delta_r(0)$ to $M$, then the Kobayashi pseudometric $k_M$ is defined on $TM$ by
$$
k_M(z;v):=\inf\{1/r > 0:\ \exists f \in \mathcal H(\Delta_r(0),M),\ f(0) = z, \ f'(0) = v\}.
$$

A complex manifold $M$ of complex dimension $n$ is {\sl Kobayashi hyperbolic} if for every point $p \in M$, there is a neighbourhood $U$ of $p$ in $M$ and a constant $c > 0$ such that $k_M(z,v) \geq c \norm{v}_g$ for every $z \in U$ and every $v \in T_zM$, where $\norm{\cdot}_g$ is any Hermitian norm on $U$ induced from $\mathbb C^n$. If $K_M$ denotes the inner distance induced by $k_M$, then $M$ is Kobayashi hyperbolic if $K_M$ is a distance on $M$. Notice that the topology induced by $K_M$ on $M$ is then equivalent to the natural topology of $M$. From the definition of $k_M$ we see that every bounded domain in $\mathbb C^n$ is Kobayashi hyperbolic, whereas a complex manifold containing a nonconstant entire curve is not Kobayashi hyperbolic. Since the Kobayashi metric is a biholomorphic invariant, it follows that a complex manifold that is not Kobayashi hyperbolic does not admit any bounded representation, i.e., is not biholomorphic to any bounded domain in $\mathbb C^n$. The first purpose of the paper is to give a sufficient condition from pluripotential theory for an unbounded domain to be Kobayashi hyperbolic.
For $r > 0$ and $z \in \mathbb C^n$, $n \ge 1$, we denote by $B^n_r(z)$ the Euclidean open ball centered at $z$ with radius $r$, i.e. $B^n_r(z):=\{w \in \mathbb C^n:\ \norm{w-z} < r\}$ where $\norm{\cdot}$ denotes the Euclidean norm in $\mathbb C^n$; in particular, $\Delta_r(z):=B^1_r(z)$. Finally, if $D$ is a domain in $\mathbb C^n$, we denote by $\partial D$ its Euclidean boundary.

\begin{definition}\label{antipeak-def}
Let $\Omega$ be an unbounded domain in $\mathbb C^n$. A bounded continuous positive plurisubharmonic (for short, psh.) function $\varphi$ on $\Omega$ will be called {\sl strong antipeak at infinity} for $\Omega$ if $\lim_{\norm{z} \rightarrow \infty} \varphi(z) = 0$.
\end{definition}

The first result of the paper is the following
\begin{theorem}\label{hyp-thm}
Let $\Omega$ be an unbounded domain in $\mathbb C^n$. If $\Omega$ has a strong antipeak function at infinity, then $\Omega$ is Kobayashi hyperbolic.
\end{theorem}

The next statement is a direct consequence of Theorem~\ref{hyp-thm}:

\begin{corollary}\label{holom-antipeak} Let $\Omega$ be an unbounded domain in $\mathbb C^n$. If there is a bounded holomorphic function on $\Omega$ that never
takes zero value and decays to 0 as $\|z\| \to \infty$, then $\Omega$ is Kobayashi hyperbolic.
\end{corollary}

Notice that the assumption of boundedness in the last two statements is essential. This can be seen from the following example.

\begin{example}\label{boundedness}
For $\varepsilon > 0$, let $\Omega_\varepsilon := \{(z,w) \in \mathbb C^2 : |w| < |z|, |z| > \varepsilon\}$. Then $\varphi: z \mapsto 1/|z|$ is a positive plurisubharmonic function on $\Omega_\varepsilon$ such that $\varphi (z) \to 0$ as $|z|^2 + |w|^2 \to \infty$. By Theorem \ref{hyp-thm}, 	$\Omega_\varepsilon$ is Kobayashi hyperbolic when $\varepsilon > 0$, but $\Omega_0$ is not because it contains the punctured line $\{(z,w) \in \mathbb C^2 : w = 0, |z|>0\}$.
\end{example}

Note that not every Kobayashi hyperbolic domain admits a bounded representation. One of the obstructions for the existence of such representations was introduced and studied by  T.Harz, N.Shcherbina and G.Tomassini (see ~\cite{HST17, HST18, HST19}) and later also by N.Shcherbina and E.Poletsky \cite{PS19}. It was named the core of a domain and can be defined as follows.

\begin{definition}
Let $\Omega$ be an unbounded domain in $\mathbb C^n$. The core $\mathfrak c(\Omega)$ is defined by

\begin{align*}
\mathfrak c(\Omega):= \{z \in \Omega:\  {\rm every} \ {\rm bounded} \ {\rm continuous} \ {\rm plurisubharmonic}\ {\rm function} \\
\ {\rm on} \ {\Omega} \  {\rm fails}\ {\rm to}\ {\rm be}\ {\rm smooth}\ {\rm and}\ {\rm strictly}\ {\rm plurisubharmonic}\ {\rm near}\ z\}.
\end{align*}
\end{definition}

Since the function $z \mapsto \norm{z}^2$ is strictly plurisubharmonic in $\mathbb C^n$, every bounded domain in $\mathbb C^n$ has an empty core. It follows from the biholomorphic invariance of the core that an unbounded domain with a nonempty core will not admit any bounded representation. For instance, in \cite[Theorem~1.2]{HST12}, the authors construct for every $n \geq 2$ an unbounded strictly pseudoconvex domain $\Omega \subset \mathbb C^n$ with smooth boundary such that $\mathfrak c(\Omega)$ is not empty and contains no analytic variety of positive dimension. Another surprising example of a strictly pseudoconvex domain in $\mathbb C^2$ with smooth boundary and nonempty core which is Kobayashi and Bergman complete, but has no nonconstant holomorphic functions, was constructed recently in \cite{SZ19}. 
As pointed out by the referee, the seminal paper \cite{Gr77} presents examples of domains in $\mathbb C^n$ which are Kobayashi hyperbolic, but do not have any bounded representation. Indeed, if $D$ denotes the complement of $(2n+1)$ hyperplanes in general position in $\mathbb C \mathbb P^n$, then $D$ is Kobayashi hyperbolic according to \cite{Gr77}. Moreover, we may assume that one of the hyperplanes is the hyperplane at infinity, and hence that $D$ is contained  in $\mathbb C^n$. If there were a biholomorphism $\Phi$ from $D$  to a bounded domain $\Omega$ in $\mathbb C^n$, then the hyperplanes would be removable singularities for $\Phi$, since $\Phi$ would be bounded in a neighborhood of the hyperplanes. Hence $\Phi$ would extend to $\mathbb C^n$ as a bounded map. Then it would be constant according to the Liouville Theorem.

The second goal of the present paper is to construct an unbounded pseudoconvex domain in $\mathbb C^3$, whose boundary is globally defined by a graph, which is Kobayashi hyperbolic and has a nonempty core. More precisely, we say that a domain $\Omega \subset \mathbb C^n$ is rigid if there exists a function $\Psi$ defined in $\mathbb C^{n-1}$ such that
$$
\Omega=\{(z,\zeta) \in \mathbb C^{n-1} \times \mathbb C: {\rm Re}(\zeta) > \Psi(z)\}.
$$
The domain $\Omega$ is pseudoconvex if and only if the function $\Psi$ is plurisubharmonic in $\mathbb C^{n-1}$. Rigid domains appear naturally as local models for pseudoconvex domains and reflect the geometry of such domains at some boundary points. For instance, the strictly pseudoconvex domain $\Omega:=\{(\zeta,z) \in \mathbb C^n:{\rm Re}(\zeta) > \|z\|^2\}$, unbounded representation of the unit ball in $\mathbb C^n$, is  a local model for domains in $\mathbb C^n$ near every strictly pseudoconvex boundary point. Likewise, if $D \subset \mathbb C^2$ is a bounded domain with smooth boundary of finite D'Angelo type $2m$ at $p \in \partial D$ (see~\cite{DAN79} for the definition of the D'Angelo type), then there are a neighbourhood $U$ of $p$ in $\mathbb C^2$ and holomorphic coordinates $(\zeta,z)$ defined on $U$ such that
$$
D \cap U=\{(z,\zeta) \in U:{\rm Re}(\zeta) > H(z)+ \phi(z,\zeta)\},
$$
where $H$ is a subharmonic homogeneous polynomial of degree $2m$ which is not harmonic and $|\phi(z,\zeta)| \leq c(|\zeta|^2 + |\zeta| |z| + |z|^{2m+1})$ on $U$. Notice that if $\Omega_H:=\{(z,\zeta) \in \mathbb C^2:{\rm Re}(\zeta) > H(z)\}$, then the metric space $(\Omega_H,K_{\Omega_H})$ is complete. Indeed, since $H$ is homogeneous, there is a sequence of automorphisms of $\Omega_H$ that accumulates at the origin. Moreover, according to the {\sc Main Theorem} in~\cite{BF78}, there is a global holomorphic peak function at the origin for $\Omega_H$, i.e. a holomorphic function $f$ from $\Omega_H$ to $\Delta_1(0)$, continuous on $\overline{\Omega}_H$, such that $f(0) = 1$ and for every bounded open neighbourhood $U$ of the origin in $\mathbb C^2$, $\sup_{\overline{\Omega} \setminus U}|f| < 1$. Notice that, by construction, $f(\Omega_H) \subset \Delta_1(0) \setminus \{0\}$ and $\lim_{\norm{p} \rightarrow \infty} f(p) = 0$. The completeness of the metric space $(\Omega_H,K_{\Omega_H})$ follows now from Proposition 3.1.4 in~\cite{G99}.

\vspace{1mm}
Observe, moreover, that $\Omega_H$ has an empty core. Indeed, assume to get a contradiction that $\mathfrak{c}(\Omega_H) \neq \emptyset$. Since $\Omega_H$ is a pseudoconvex domain of finite type in $\mathbb C^2$, it admits a local holomorphic peak function at each boundary point. It follows then that $\overline{\mathfrak{c}(\Omega_H)} \cap \partial \Omega_H = \emptyset$. Moreover, we know by Theorem II in \cite{PS19} (see also Theorem 3.3 in \cite{Sl19}) that the set $\mathfrak{c}(\Omega_H)$ is the disjoint union of  some closed sets $E_j, j \in J$, that are 1-pseudoconcave in the sense of Rothstein and have the following Liouville property: every bounded continuous psh. function on $\Omega$ is constant on each of $E_j$. Let $E_{j_0}$ be one of the sets in the decomposition above. Then, in view of the 1-pseudoconcavity of $E_{j_0}$, $E_{j_0}$ is unbounded. Since $|f|^2$ is a bounded continuous psh. function on $\Omega_H$, the restriction of $|f|^2$ to $E_{j_0}$ is constant. Hence, it follows from the fact that $\lim_{\norm{p} \rightarrow \infty} f(p) = 0$ that $f$ vanishes identically on $E_{j_0}$. This is a contradiction since $f$ does not vanish on $\Omega_H$.

\vspace{2mm}
It was proved in a recent paper \cite{Sh19} that the existence of the Kobayashi and the Bergman metrics for pseudoconvex domains in $\mathbb C^2$ of more general form
$$
\Omega_H:=\{(z,\zeta) \in \mathbb C^2:{\rm Re}(\zeta) > H(z, {\rm Im}(\zeta))\}, 
$$
with $H$ being just a continuous function on $\mathbb C \times \mathbb R$, is equivalent to the fact that the core $\mathfrak c(\Omega_H)$ of $\Omega_H$ is empty.

The second result of the present paper shows that this kind of relations does not hold in the case of higher dimensions.

\begin{theorem}\label{main-thm}
There exists a nonnegative plurisubharmonic function $\Psi$ in $\mathbb C^2$ such that the rigid domain
$$
\Omega_{\Psi}:=\{(z,w,\zeta) \in \mathbb C^3:{\rm Re}(\zeta) > \Psi(z,w)\}
$$
is Kobayashi hyperbolic and has a nonempty core.
In particular, the domain $\Omega_{\Psi}$ is not biholomorphic to a bounded domain.
\end{theorem}

The following corollary shows that the existence of a  strong antipeak function at infinity is not a necessary condition for an unbounded domain to be Kobayashi hyperbolic. Indeed, from the construction of $\Omega_{\Psi}$ in Theorem~\ref{main-thm}, we have
\begin{corollary}\label{main-cor}
The domain $\Omega_{\Psi}$ does not admit any strong antipeak function at infinity.
\end{corollary}

The paper is organized as follows. In Section~\ref{hyp-sect} we prove Theorem~\ref{hyp-thm}. In Section~\ref{wermer-sect} we construct explicitly the function $\Psi$ used in Theorem~\ref{main-thm} and a Wermer type set contained in $\Omega_{\Psi}$. Finally, in Section~\ref{main-sect} we prove Theorem~\ref{main-thm} and Corollary~\ref{main-cor}.

\vspace{2mm}
{\sl Acknowledgments.} The authors wish to thank the anonymous referee for precious comments and suggestions, which improved the content of the paper. In particular, Corollary \ref{holom-antipeak} and Example \ref{boundedness} were proposed by the referee. Lemma \ref{loc-lem} and its proof were also modified following the referee's suggestion, as well as different other points all along the paper.

Part of this work was done while the second author was a visitor at the Capital Normal University (Beijing). It is his pleasure to thank this institution for its hospitality and good working conditions. The authors also would like to thank Fusheng Deng for his remark related to the definition of the antipeak function which slightly strengthen the statement of Theorem 1.

\section{Proof of Theorem~\ref{hyp-thm}}\label{hyp-sect}

\vskip 0,2cm
We first notice that a domain $\Omega \subset \mathbb C^n$ is Kobayashi hyperbolic if and only if it satisfies the following condition: 
 
\begin{equation}\label{hyp-eq} 
\forall p \in \Omega,\ \exists r > 0, \exists c > 0 / \ \forall q \in B^n_r(p), \ \forall v \in \mathbb C^n,\ k_\Omega(q,v) \geq c \|v\|,
\end{equation} 
where $\|v\|$ denotes the Euclidean norm in $\mathbb C^n$.

Let now $\Omega$ be a domain satisfying the assumptions of Theorem~\ref{hyp-thm}. Assume, to get a contradiction, that $\Omega$ is not Kobayashi hyperbolic. It follows from (\ref{hyp-eq}) that there is a point $p \in \Omega$ and for every positive integer $k$ there is a holomorphic map $f_k: \Delta_k(0) \rightarrow \Omega$ such that $\|f'_k(0)\|=1$ and the sequence $\{f_k(0)\}_k$ converges to $p$ when $k$ goes to infinity.
Moreover, we may assume that $f_k$ is continuous up to $\partial \Delta_k(0)$. 
 
Denote by $\varphi$ a psh. function on $\Omega$ that is strong antipeak at infinity. Let $C > 0$ be a constant which bounds the function $\varphi$ from above, i.e. $\varphi < C$ on $\Omega$. Observe that, in view of the continuity of $\varphi$, there is a positive constant $\alpha$ such for sufficiently large $k$ we have: 
$$ 
\varphi(f_k(0)) \geq \alpha. 
$$ 
 
For each $R > 0$ we let $c_R:=\sup_{\mathbb C^n \backslash {B^n_R(0)}}\varphi$. Notice that, by Definition~\ref{antipeak-def}, we get: $\lim_{R \rightarrow \infty}c_R = 0$. Since the Euclidean ball $B^n_R(0)$ is Kobayashi hyperbolic, it follows that for every sufficiently large positive integer $k$ we have: 
$$ 
f_{k}(\partial \Delta_k(0)) \cap \left(\mathbb C^n\backslash \overline{B^n_R(0)}\right) \neq \emptyset. 
$$ 
 
Let $U_{k,R}$ denote the connected component of $f_k^{-1}(f_k(\Delta_k(0)) \cap B^n_R(0))$ containing the origin.

\vskip 0,2cm
\noindent{\bf {Claim 1.}} For each $k \in \mathbb N$ and $R>0$, the domain $U_{k,R}$ is simply connected.

\vskip 0,2cm
Indeed, if $\partial U_{k,R}$ has at least two components, then, after maybe substituting $R$ with a generic value $\tilde{R} < R$, there is a disc $V_{k,R}$, contained in $\Delta_k(0)$, such that $f_k(V_{k,R}) \subset \Omega \backslash B^n_R(0)$ and such that $f_k(\partial V_{k,R}) \subset \partial B^n_R(0)$. Let $R'=\sup_{V_{k,R}}|f_k|$. Then the complex disc $f_k(V_{k,R})$ is tangent to $\partial B^n_{R'}(0)$ from inside which is not possible. This proves Claim 1.

Hence $U_{k,R}$ is bounded by a piecewise smooth Jordan curve for generic values of $R$. We denote by $\Phi$ a Riemann map from $\Delta_1(0)$ to $U_{k,R}$ such that $\Phi(0)=0$.
According to the Carath{\'e}odory Theorem, $\Phi$ extends as a homeomorphism between $\partial \Delta_1(0)$ and $\partial U_{k,R}$.
 
\vskip 0,2cm
\noindent{\bf Claim 2.} There is $R_0 > 0$ such that for each $k \in \mathbb N$ and each $R > R_0$ one has $\overline{U}_{k,R} \cap \partial \Delta_k(0) \neq \emptyset$.

\vskip 0,2cm
Indeed, assume to get a contradition that $\partial U_{k,R}$ is a closed curve contained in $\Delta_k(0)$. Then $f_k(\Phi(e^{i\theta})) \in \partial B^n_R(0)$ for all $\theta \in [0,2\pi]$ and from the Mean Value Inequality
\begin{equation}\label{mv-eq} 
(\varphi \circ f_k \circ \Phi)(0) \leq \frac{1}{2\pi}\int_0^{2\pi} \varphi(f_k(\Phi(e^{i\theta})))d\theta
\end{equation} 

\noindent 
we get the inequality
$$
\alpha \leq c_R.
$$
If we choose $R_0$ so large that $c_{R_0} < \alpha$, then we get a contradiction. This proves Claim 2.

\vskip 0,2cm
Let $\omega(0,\overline{U}_{k,R} \cap \partial \Delta_k(0), U_{k,R})$ denote the harmonic measure of $\overline{U}_{k,R} \cap \partial \Delta_k(0)$ at the point $0$ with respect to the domain $U_{k,R}$. We recall that if $D$ is a bounded domain in $\mathbb C$, $p \in D$ and $E$ is a Borel set in $\partial D$, then $\omega(p,E,D)$ denotes the harmonic measure given by the value at $p$ of the solution to the Dirichlet problem on $D$, whose boundary value on $\partial D$ is equal to the characteristic function of $E$. Then we have
$$
\omega(0,\overline{U}_{k,R} \cap \partial \Delta_k(0), U_{k,R}) =\frac{l(\Phi^{-1}(\overline{U}_{k,R} \cap \partial \Delta_k(0)))}{2\pi},
$$
where $l(\Phi^{-1}(\overline{U}_{k,R} \cap \partial \Delta_k(0)))$ denotes the length of the set $\Phi^{-1}(\overline{U}_{k,R} \cap \partial \Delta_k(0))$. Since $U_{k,R}$ is simply connected and contained in $\Delta_k(0)$, it follows from Lemma 3.4 of \cite{Sh93} that the Euclidean distance $\rho(0,\partial U_{k,R})$ from $0$ to the boundary of $U_{k,R}$ satisfies:

$$\rho(0,\partial U_{k,R}) \geq \frac{\pi^2 k}{16}\left(\omega(0,\overline{U}_{k,R} \cap \partial \Delta_k(0), U_{k,R})\right)^2.
$$

To finish the proof of Theorem~\ref{hyp-thm} we consider two cases.

\vskip 0,3cm
\noindent{\bf Case 1.} {\sl There exist $R > 0$ and $c > 0$ such that $\omega(0,\overline{U}_{k_p,R} \cap \partial \Delta_{k_p}(0),U_{k_p,R})\geq c$ for some sequence of positive integers $k_p$ with $k_1 < k_2 < k_3 < \cdots$.} 

\vskip 0,2cm
In this case we conclude from the previous inequality that for these numbers $k_p$ one has:

\begin{equation}\label{harm-eq}
\rho(0,\partial U_{k_p,R}) \geq \frac{\pi^2 k_p}{16}c^2=c^* k_p,
\end{equation}
where $c^*=\frac{\pi^2 c^2}{16}>0$.

Hence the ball $B^n_R(0)$ contains the set $f_{k_p}(\Delta_{c^*k_p})$ for arbitrarily large numbers $k_p$. This contradicts the Kobayashi hyperbolicity of $B^n_R(0)$.

\vskip 0,2cm
\noindent{\bf Case 2.} For each $R > 0$ one has $\lim_{k \rightarrow \infty}\omega(0,\overline{U}_{k,R} \cap \partial \Delta_k(0),U_{k,R}) = 0$.

\vskip 0,2cm
We first observe that for each $\varepsilon >0$ we have $\omega(0,\partial U_{k,R} \cap \Delta_k(0),U_{k,R}) \geq 1 - \varepsilon \,$ for every sufficiently large $k$. Since $f_k(\partial U_{k,R}  \cap \Delta_k(0)) \subset \partial B^n_R(0)$, we conclude that
$$
\sup_{f_k({\partial U_{k,R} \cap \Delta_k(0)})}\varphi \leq c_R.
$$ 
It follows also from the choice of $C$ that
$$ 
\sup_{f_k({{\bar U}_{k,R} \cap \partial \Delta_k(0)})}\varphi \leq C.
$$
Then the Mean Value Inequality (\ref{mv-eq})
implies the inequality
$$
\alpha \leq C \varepsilon + {c_R} (1 - \varepsilon).
$$
If we choose now $\varepsilon$ so small that $C \varepsilon < \frac{\alpha}{2}$ and then $R$ so large that ${c_R} (1 - \varepsilon) < \frac{\alpha}{2}$, then we get a contradiction. This concludes the proof of Theorem~\ref{hyp-thm}. \qed

\vskip 0,2cm
\noindent{\bf {Remark 1.}} If the domain $\Omega$ is, moreover, assumed to be strictly pseudoconvex, then we can also give a completely different proof of Theorem~\ref{hyp-thm} which uses recent nontrivial results on the structure of the core obtained in \cite{HST18}, \cite{PS19} and \cite{Sl19}.
Indeed, we first prove the following

\vspace{2mm}
\noindent
{\bf Claim.} The core $\mathfrak{c}(\Omega)$ is empty.

\vspace{2mm}
\noindent \textbf{Proof of the Claim.} The argument here is similar to the one used to prove that $\mathfrak{c}(\Omega_H) = \emptyset$ in the example considered in the introduction. 
Assume to get a contradiction that $\mathfrak{c}(\Omega) \neq \emptyset$ and let $E_{j_0}$ be one of the closed 1-pseudoconcave sets in the decomposition of $\mathfrak{c}(\Omega)$ granted by Theorem II in \cite{PS19}. Then the restriction of the antipeak function $\varphi$ to $E_{j_0}$ is constant. In view of strict pseudoconvexity of $\Omega$, one has that $\mathfrak{c}(\Omega) \cap \partial \Omega = \emptyset$. Then, since $E_{j_0}$ is 1-pseudoconcave in the sense of Rothstein, we conclude that the set $E_{j_0}$ has to be unbounded. It follows now from the requirement $\lim_{\norm{z} \rightarrow \infty} \varphi(z) = 0$ on the strong antipeak function $\varphi$ that $\varphi \equiv 0$ on $E_{j_0}$, which is impossible by the definition of the antipeak function, since $\varphi$ is positive on $\Omega$. This proves the Claim.  $\hfill \Box$

\vspace{1mm}
Now, using the argument of Lemma 1 in \cite{Sh19}, we get a bounded continuous strictly psh. function $\phi$ on $\Omega$. It follows then from Theorem 3 on p. 362 in \cite{Si81} that $\Omega$ is Kobayashi hyperbolic.

\vskip 0,1cm
We do not know if a similar argument can also be applied for general (not necessarily strictly pseudoconvex) unbounded domains in ${\mathbb C}^n$.

\vskip 0,3cm
\noindent{\bf Remark 2.} A weaker notion of an antipeak function was introduced and studied in \cite{G99}. That notion is not strong enough to guarantee the claim of Theorem~\ref{hyp-thm} as it can be seen from the following example:

\vskip 0,1cm
Let $\,\,\, \Omega:= \big\{(z,w) \in {\mathbb C}^2: \log\abs{w} + \big(\abs{z}^2 + \abs{w}^2\big) < 0 \big\} \subset {\mathbb C}^2$ and let $\varphi(z, w):= - \log\abs{z}$. It is easy to see that $\varphi$ is an antipeak function for $\Omega$ in the sense of \cite{G99}, but $\Omega$ is obviously not Kobayashi hyperbolic due to the fact that $\{w = 0\} \subset \Omega$.

\vskip 0,1cm
We do not know if for an unbounded domain $\Omega$ in ${\mathbb C}^n$ (which we can assume in addition to be pseudoconvex or, even, strictly pseudoconvex) the fact that $\mathfrak{c}(\Omega) = \emptyset$ implies that there is a strong antipeak function at infinity for $\Omega$.

\vspace{3mm}
Finally, we point out that in the paper \cite{NH08}, the authors considered (non necessarily continuous) bounded above psh. functions $\phi$ defined on some unbounded domains in $\mathbb C^n$ and having the property $\lim_{\norm{z} \rightarrow \infty}\phi(z) = -\infty$ with the aim to study a Dirichlet type problem on some family of unbounded domains.

\section{Construction of $\Psi$ and of a Wermer type set in $\Omega_{\Psi}$}\label{wermer-sect}

Let $\{a_n, n \in \mathbb N\}$ be the sequence of points with entire coordinates in $\mathbb C$ such that $a_1=0$ and, for every $n \in \mathbb N$, the set $ (\Zb + i\Zb) \cap \{ \zeta \in \mathbb C : -n \leq {\rm Re}(\zeta), {\rm Im}(\zeta) \leq n\}$ consists of the points $a_1, \dots, a_{(2n+1)^2}$. We may select the points to form a spiral turning anticlockwise, starting with $a_2:(1,0)$, $a_3:(1,1)$, ... (See Figure~\ref{Fig1}.)

 \begin{figure}[h!]


\begin{center}
\begin{tikzpicture}

\draw[->] [dashed] (0,-1.5) -- (0,1.5) ;
\draw[->] [dashed] (-1.5,0) -- (2.5,0) ;

\draw  (0,0) -- (1,0) ;

\draw  (1,0) -- (1,1) ;

\draw  (1,1) -- (0,1) ;

\draw  (0,1) -- (-1,1) ;

\draw  (-1,1) -- (-1,0) ;

\draw  (-1,0) -- (-1,-1) ;

\draw  (-1,-1) -- (0,-1) ;

\draw  (0,-1) -- (1,-1) ;

\draw  (1,-1) -- (2,-1) ;

\draw  (2,-1) -- (2,0) ;

\draw  (2,0) -- (2,1) ;

\draw (2,1) -- (2,2);

\draw (2,2) -- (1,2);

\draw (1,2) -- (0,2);

\draw[dotted] (0,2) -- (-0.8,2);

\draw (0.5,0) node[scale=0.65] {$>$};

\draw (1,0.5) node[scale=0.65] {$\wedge$};

\draw (0.5,1) node[scale=0.65] {$<$};

\draw (-0.5,1) node[scale=0.65] {$<$};

\draw (-1,0.5) node[scale=0.65] {$\vee$};

\draw (-1,-0.5) node[scale=0.65] {$\vee$};

\draw (-0.5,-1) node[scale=0.65] {$>$};

\draw (0.5,-1) node[scale=0.65] {$>$};

\draw (1.5,-1) node[scale=0.65] {$>$};

\draw (2,-0.5) node[scale=0.65] {$\wedge$};

\draw (2,0.5) node[scale=0.65] {$\wedge$};

\draw (2,1.5) node[scale=0.65] {$\wedge$};

\draw (1.5,2) node[scale=0.65] {$<$};

\draw (0.5,2) node[scale=0.65] {$<$};


\draw (0,0) node[scale=0.65] {$\bullet$};
\draw (0.25,-0.2) node[scale=1] {$a_1$};

\draw (1,0) node[scale=0.65] {$\bullet$};
\draw (1.2,-0.2) node[scale=1] {$a_2$};

\draw (1,1) node[scale=0.65] {$\bullet$};
\draw (1.2,1.2) node[scale=1] {$a_3$};

\draw (0,1) node[scale=0.65] {$\bullet$};
\draw (0.25,1.2) node[scale=1] {$a_4$};

\draw (-1,1) node[scale=0.65] {$\bullet$};
\draw (-1.2,1.2) node[scale=1] {$a_5$};

\draw (-1,0) node[scale=0.65] {$\bullet$};
\draw (-1.2,-0.2) node[scale=1] {$a_6$};

\draw (-1,-1) node[scale=0.65] {$\bullet$};
\draw (-1.2,-1.2) node[scale=1] {$a_7$};

\draw (0,-1) node[scale=0.65] {$\bullet$};
\draw (0.25,-1.2) node[scale=1] {$a_8$};

\draw (1,-1) node[scale=0.65] {$\bullet$};
\draw (1.2,-1.2) node[scale=1] {$a_9$};

\draw (2,-1) node[scale=0.65] {$\bullet$};
\draw (2.3,-1.2) node[scale=1] {$a_{10}$};

\draw (2,0) node[scale=0.65] {$\bullet$};
\draw (2.35,-0.2) node[scale=1] {$a_{11}$};

\draw (2,1) node[scale=0.65] {$\bullet$};
\draw (2.35,1.2) node[scale=1] {$a_{12}$};

\draw (2,2) node[scale=0.65] {$\bullet$};
\draw (2.35,2.2) node[scale=1] {$a_{13}$};

\draw (1,2) node[scale=0.65] {$\bullet$};
\draw (1.35,2.2) node[scale=1] {$a_{14}$};

\draw (0,2) node[scale=0.65] {$\bullet$};
\draw (0.35,2.2) node[scale=1] {$a_{15}$};


























\end{tikzpicture}
\end{center}
  \caption{}\label{Fig1}
\end{figure}

\subsection{Construction and properties of a Wermer type set in $\Omega_{\Psi}$}\label{wermer-subsection} We consider the following Wermer type set, whose construction is similar to the one used in \cite{HST12}. Let $\{\varepsilon_n\}_{n \in \mathbb N}$ be a sequence of positive numbers, decreasing to zero. First conditions on the speed of convergence of $\{\varepsilon_n\}$ are provided by Lemma 2.2 in~\cite{HST12}. Then, for every $n \in \mathbb N$, let
$$
E_n:=\{(z,w) \in \mathbb C^2:\ w=\sum_{j=1}^n \varepsilon_j \sqrt{z-a_j}\}.
$$
Following~\cite{HST12}, we define the Wermer type set

$$
\mathcal E:= \cup_{R>0} \left(\lim_{n \rightarrow \infty}\left(E_n \cap \overline{B^2_R(0)}\right)\right) \subset \mathbb C^2,
$$
where the limit is understood with respect to the Hausdorff distance.

Moreover, the same argument as in  Lemma 5.1 of~\cite{HST12} shows that there exists a psh. function $\phi:\mathbb C^2 \rightarrow [-\infty,+\infty)$ such that $\mathcal E=\{\phi = -\infty\}$ and $\phi$ is pluriharmonic on $\mathbb C^2 \backslash \mathcal E$.

For $\rho:[0,+\infty) \rightarrow \mathbb R$, let $\Psi(\rho)$ be the function defined on $\mathbb C^2$ by
\begin{equation}\label{function-eq}
\Psi(\rho)(z,w): = e^{\phi(z,w) + \rho(|{\rm Re}(z)|) + \rho(|{\rm Im}(z)|)}.
\end{equation}

To prove Theorem~\ref{main-thm} we will construct a convex function $\rho:[0,+\infty) \rightarrow [0,+\infty)$, with $\lim_{x \rightarrow +\infty}\rho(x) = +\infty$, such that the function $\Psi(\rho)$ satisfies the conclusion of Theorem~\ref{main-thm}.

\vspace{2mm}
But first we prove the following property of the domain $\Omega_{\Psi(\rho)}$ which is claimed in Theorem~\ref{main-thm}.
\begin{lemma}\label{core-lem}
The core $\mathfrak c(\Omega_{\Psi(\rho)})$ is not empty. In particular, $\Omega_{\Psi(\rho)}$ is not biholomorphic to a bounded domain in $\mathbb C^3$.
\end{lemma}

\begin{proof}
It follows from the same argument as in ~\cite[Theorem~2.2]{HST18} that the set $\mathcal E \times \{1\}$, which is contained in $\Omega_{\Psi(\rho)}$, satisfies the Liouville type property, meaning that every continuous psh. function defined in a neighbourhood of $\mathcal E \times \{1\}$ and bounded there from above is constant on $\mathcal E \times \{1\}$. Hence, $\mathcal E \times \{1\}$ is contained in $\mathfrak c(\Omega_{\Psi(\rho)})$.
\end{proof}

Let $\mathcal F_{conv}$ denote the set of convex functions $\rho:[0,+\infty) \rightarrow [0,+\infty)$, with $\lim_{x \rightarrow +\infty}\rho(x) = +\infty$. 
For $\rho \in \mathcal F_{conv}$, let $(z_0,w_0,\zeta_0) \in \Omega_{\Psi(\rho)}$ and let $U \ni (z_0,w_0,\zeta_0)$ be a neighbourhood of $(z_0,w_0,\zeta_0)$, relatively compact in $\Omega_{\Psi(\rho)}$. We first prove a localization lemma that will be used for a sequence of ``large'' holomorphic disks with center in $U$. The present form of Lemma~\ref{loc-lem} and its proof were suggested by the referee. Let $\pi_{z,w}: \mathbb C^3_{z,w,\zeta} \rightarrow \mathbb C^2_{z,w}$ denote the canonical projection. 

\begin{lemma}\label{loc-lem}
Let $r > 0$ and let $f:\Delta_r(0) \rightarrow \Omega_{\Psi(\rho)}$ be holomorphic. Then
$$
\pi_{z,w}\left(f\left(\Delta_{r/3}(0)\right)\right) \subset \left\{(z,w) \in \mathbb C^2:\ \Psi(\rho)(z,w) < 2{\rm Re}(\zeta(0))\right\}.
$$
\end{lemma}

\begin{proof}
In view of the harmonicity and positivity of the function $\lambda \mapsto {\rm Re}(\zeta(\lambda))$ on $\Delta_r(0)$ and Harnack's inequality (see Theorem 1.4.1 in~\cite{AG01}), we have: $2{\rm Re}(\zeta(0)) > {\rm Re}(\zeta(\lambda))$ on $\Delta_{r/3}(0)$. Hence, by the definition of $\Omega_{\Psi(\rho)}$, we have on $\Delta_{r/3}(0)$:
$$
\Psi(\rho)(z(\lambda),w(\lambda)) < 2{\rm Re}(\zeta(0)).
$$
\end{proof}

\subsection{Construction of the convex function $\rho$} We now construct the function $\rho$ that will enter the definition of the function $\Psi$ in Theorem~\ref{main-thm}.

\begin{lemma}\label{conv-lem}
Let $\{c(n),\ n \geq 0\}$ be an arbitrary increasing sequence of positive numbers. Then there is a strictly convex function $\rho$ of class $\mathcal C^{\infty}$ on $[0,+\infty)$ such that $\rho'(0) = 0$ and for every $n \geq 0$ and every $t \in (n,n+1]$, one has
$$
\rho(t) > c(n).
$$
\end{lemma}

\begin{proof}
We first construct inductively an auxiliary convex function $\rho_1$ such that for every $n \geq 0$,\,\, $\rho_1$ is affine on the segment $[n,n+1]$. For $n=0$ we set ${\rho_1}\bigr|_{[0,1]} = c(1)$. Let $n \in \mathbb N$ and assume that $\rho_1$ is already constructed on $[0,n]$. In particular, there exist $a_n,b_n > 0$ such that for every $t \in (n-1,n]$
$$
\rho_1(t) = a_n t + b_n.
$$
\begin{itemize}
\item If $a_n (n+1) + b_n \geq c(n+1)$, we set $\rho_1(t) = a_n t + b_n$ for every $t \in (n,n+1]$,
\item If $a_n (n+1) + b_n < c(n+1)$, we set $\rho_1(t) = (a_n n + b_n)(1-(t-n)) + c(n+1) (t-n)$ for every $t \in (n,n+1]$.
\end{itemize}
This defines the function $\rho_1$ on $[0,+\infty)$ by induction.
Let now $\chi: \mathbb R \rightarrow \mathbb R$ be a nonnegative ${\mathcal C^{\infty}}-$smooth function with support contained in $[-1/4,1/4]$ and satisfying $\int_{\mathbb R} \chi = 1$. Then the restriction to $[0,+\infty)$ of the function $\rho$ defined on $\mathbb R$ by $\rho(t):=\tilde{\rho}_1 \ast \chi(t) + t^2$, where $\tilde{\rho}_1(t) = \rho_1(|t|)$ for every $t \in \mathbb R$, will satisfy all the conditions of Lemma~\ref{conv-lem}.
\end{proof}

For every $n \in \mathbb N$, let
\begin{multline}\label{sn-eq}
S_n:=\left\{z \in \mathbb C:\ -\left(n+\frac{3}{4}\right) \leq {\rm Re}(z), {\rm Im}(z) \leq n + \frac{3}{4}\right\} \backslash \\
\backslash \left\{z \in \mathbb C:\ -\left(n+\frac{1}{4}\right) < {\rm Re}(z), {\rm Im}(z) < n + \frac{1}{4}\right\}
\end{multline}
and
\begin{multline}\label{tn-eq}
T_n:=\left\{z \in \mathbb C:\ {\rm Re}(z)= \pm \left(n+\frac{1}{2}\right),\ |{\rm Im}(z)| \leq n + \frac{1}{2}\right\} \cup \\
\cup \left\{z \in \mathbb C:\ |{\rm Re}(z)| \leq n+\frac{1}{2},\ {\rm Im}(z) = \pm \left(n + \frac{1}{2}\right)\right\}.
\end{multline}

(See Figure~\ref{Fig2}.)

 \begin{figure}[h!]


\begin{center}
\begin{tikzpicture}[scale=1.5]


\draw[color=black,fill=black!10] (-1.3,-1.3) -- (1.3,-1.3) -- (1.3,1.3) -- (-1.3,1.3) -- cycle;

\draw[color=black,fill=white] (-0.7,-0.7) -- (0.7,-0.7) -- (0.7,0.7) -- (-0.7,0.7) -- cycle;

\draw[ultra thick][dashed] (-1,-1) -- (1,-1) -- (1,1) -- (-1,1) -- cycle;

\draw[->] [dashed] (0,-1.7) -- (0,1.7) ;
\draw[->] [dashed] (-1.7,0) -- (1.7,0) ;

\draw[->] (-2.5,0) -- (-1.01,-0.7);
\draw (-2.7,0) node[scale=1] {$T_n$};

\draw[->] (2.5,0) -- (1.05,1.15);
\draw (2.7,0) node[scale=1] {$S_n$};









\end{tikzpicture}
\end{center}
  \caption{}\label{Fig2}
\end{figure}

Since, for every $n \in \mathbb N$, the set $S_n$ does not contain any point with entire coordinates, then for every $n,m \in \mathbb N$ and for every $p \in S_n$, the restriction to $\Delta_{1/4}(p)$ of the defined above set $E_m$, denoted ${E_m \vert}_{\Delta_{1/4}(p)}$, is a union of holomorphic graphs of the form $\{w=f(z)\}$. More precisely, for every $p \in S_n$, there are holomorphic functions $f_1,\dots,f_{2^m}$, defined on $\Delta_{1/4}(p)$, such that ${E_m \vert}_{\Delta_{1/4}(p)}=\cup_{1 \leq j \leq 2^m}\{w=f_j(z)\}$. Since $S_n$ is compact in $\mathbb C$, it follows from the Montel Theorem that for every $p \in S_n$, ${\mathcal E {\vert}}_{\Delta_{1/4}(p)}=\cup_{\lambda \in \mathcal A}\{w=f_{\lambda}(z)\}$, where $\mathcal A$ denotes a Cantor set parametrising the branches of $\mathcal E$ over $\Delta_{1/4}(p)$. Indeed, since the sequence $\{\varepsilon_m\}_{m \in \mathbb N}$ converges to 0 sufficiently fast, then for every sufficiently large $m$, the set $E_{m+1}$ consists of $2^{m+1}$ disks that are sufficiently small perturbations of the $2^m$ disks consituting $E_m$. Passing to the limit when $m$ goes to infinity, we obtain a structure of a Cantor set on the vertical fibers.

Moreover, for every $\lambda \in \mathcal A$, $f_{\lambda}$ is holomorphic on $\Delta_{1/4}(p)$ and, hence, in view of the compactness of both the set $S_n$ and the family $\{w=f_{\lambda}(z)\}_{\lambda \in \mathcal A}$ with respect to the parameter $\lambda$, one can define

\begin{equation}\label{sup-eq}
\alpha(n):= \sup \{ |f'_{\lambda}(z)| : \lambda \in \mathcal A, z \in \Delta_{1/8}(p), p \in S_n \} < +\infty.
\end{equation}

\begin{lemma}\label{horiz-lem}
We can choose the sequence $\{\varepsilon_n\}_{n \in \mathbb N}$ decreasing and converging to zero so fast that
$$
\lim_{n \rightarrow \infty}\alpha(n) = 0.
$$
\end{lemma}

\begin{proof}
We first point out that, since every map $f_{\lambda}$ is the uniform limit on $\Delta_{1/8}(p)$ of functions whose graph over $\Delta_{1/8}(p)$ is a branch of the multivalued holomorphic function $\sum_{k=1}^m \varepsilon_k \sqrt{z-a_k}$, it is sufficient to prove Lemma~\ref{horiz-lem} for these functions on $\Delta_{1/8}(p)$ uniformly with respect to $m \in \mathbb N$.

Let $\varepsilon > 0$. Since $\varepsilon_k$ decreases sufficiently fast to zero according to~\cite{HST12}, there exists $k_0 \geq 1$ such that $\sum_{k \geq {(2k_0+1)^2}+1} \varepsilon_k < \varepsilon/4$. Moreover, for every $n \in \mathbb N$ and every $p \in S_n$ we have
$$
\inf_{z \in \Delta_{1/8}(p)}d\left(z,\mathbb Z + i \mathbb Z\right) \geq \frac{1}{8}.
$$
Hence, for every $m > (2k_0+1)^2$, for every $n \in \mathbb N$ and every $p \in S_n$, each branch of the multivalued holomorphic function $\sum_{k=(2k_0+1)^2+1}^m \varepsilon_k \sqrt{z-a_k}$ is given by the graph of a holomorphic function such that the modulus of its derivative is bounded from above by $\varepsilon/2$ on $\Delta_{1/8}(p)$.

Now, there exists $n_0 > k_0$ such that for every $n \geq n_0$,
$$
\frac{1}{\inf \{ \sqrt{d(a_j,S_n)} : a_j \in \{\lambda \in \mathbb C:\ -k_0 \leq {\rm Re}(\lambda), {\rm Im}(\lambda) \leq k_0\}} < \frac{\varepsilon}{\sum_{j \geq 1}\varepsilon_j}.
$$
This implies that for every $n > n_0$ and for every $p \in S_n$, each branch of the multivalued holomorphic function $\sum_{k=1}^{(2k_0+1)^2} \varepsilon_k \sqrt{z-a_k}$ is the graph of a holomorphic function such that the modulus of its derivative is bounded from above by $\varepsilon/2$ on $\Delta_{1/8}(p)$.

We finally obtain that for every $\varepsilon > 0$, for every $n \geq n_0$, for every $p \in S_n$ and for every $m > (2k_0+1)^2$, every branch of the multivalued holomorphic function $\sum_{k=1}^m \varepsilon_k \sqrt{z-a_k}$ is the graph of a holomorphic function whose derivative is bounded from above by $\varepsilon$ on $\Delta_{1/8}(p)$. This completes the proof of Lemma~\ref{horiz-lem}.
\end{proof}

For every $w \in \mathbb C,\ \delta > 0$, we consider the set
$$
\mathcal E^{\delta}:=\cup_{(z,w) \in \mathcal E}\left(\{z\} \times \Delta_{\delta}(w)\right).
$$

In view of Lemma~\ref{horiz-lem}, we can define $q_0:=\inf\{k \in \mathbb N:\ \alpha(j) < \frac{1}{2}\ {\rm for} \ {\rm all}\ j \geq k\}$. Now, for every $n \geq q_0$, let $\mathcal H_{n}$ denote the set of holomorphic disks $D=\left\{(f(w),w),\ w \in \Delta_{1}(w_0)\}\right\}$ such that

\vskip 0,2cm
\begin{enumerate}
\item $(f(w_0),w_0) \in \mathcal E \cap \left(T_{n} \times \mathbb C\right)$,\\
\item for every $w \in \Delta_{1}(w_0)$, $|f'(w)| < 1$.
\end{enumerate}

\vskip 0,2cm
Let $\pi_w: \mathbb C^2_{z,w} \rightarrow \mathbb C_w$ denote the canonical projection and, for every $D \in \mathcal H_n$, let
$$
\beta^{\delta}(D):=\sup\left\{{\rm diam}(c):\ c\ {\rm connected}\ {\rm component}\ {\rm of}\ {\rm the} \ {\rm closure} \ {\rm of} \ \pi_w\left(\frac{1}{4} D \cap \mathcal E^{\delta}\right)\right\},
$$
where for a holomorphic disk $D=\left\{(f(w),w),\ w \in \Delta_{1}(w_0)\right\} \subset \mathcal H_{n}$ we denote by $\frac{1}{4} D$ the set $\frac{1}{4} D \coloneqq \left\{(f(w),w),\ w \in \Delta_{\frac{1}{4}}(w_0)\right\}.$ Here, in view of Conditions (1) and (2) above, we choose the radius equal to $1/4$ to insure that the disc $\frac{1}{4}D$ is contained in the set $S_n \times \mathbb C$. Moreover, it will also insure that the family of disks $\frac{1}{4}D_k$ considered in the proof of Lemma~\ref{vert-lem} will converge smoothly up to the boundary to $\frac{1}{4}D_{\infty}$.

\vskip 0,2cm
Finally, define
$$
\beta_{n}^{\delta}:=\sup_{D \in \mathcal H_n}\beta^{\delta}(D).
$$

\begin{lemma}\label{vert-lem}
For every $n \geq q_0$ we have
$$
\lim_{\delta \rightarrow 0}\beta_{n}^{\delta} = 0.
$$
\end{lemma}

\begin{proof}
Assume, to get a contradiction, that there exists $n \geq q_0$, a sequence of positive real numbers $\delta_k$ decreasing to 0, a sequence of points $w_k \in \pi_w\left(\mathcal E \cap \left(T_{n} \times \mathbb C\right)\right)$, a sequence of holomorphic disks $D_k=\left\{(f_k(w),w),\ w \in \Delta_{1}(w_k)\}\right\} \in \mathcal H_n$ and, for every $k \geq 1$, a connected component $c_k$ of the closure of $\pi_w(\frac{1}{4} D_k \cap \mathcal E^{\delta_k})$, such that
\begin{equation}\label{inf-eq}
\inf_{k \geq 1}{\rm diam}(c_k) =:\alpha_{\infty} > 0.
\end{equation}
We can assume that for every $k \geq 1$,  $c_k$ is simply connected.

\vskip 0,2cm
Since $\mathcal E \cap \left(T_{n} \times \mathbb C\right)$ is compact, and since $D_k \in \mathcal H_n$ for every $k \geq 1$, it follows from Condition (2) and from the Montel Theorem that up to extracting a subsequence, $D_k$ will converge to a holomorphic disc $D_{\infty}:=\{(f_{\infty}(w),w),\ w \in \Delta_1(w_{\infty})\}$ for some point $w_{\infty} \in \pi_w\left(\mathcal E \cap \left(T_{n} \times \mathbb C\right)\right)$, where $w_k \to w_\infty$ as $k \to \infty$, $f_{\infty}$ is holomorphic on $\Delta_1(w_{\infty})$ and satisfies $\sup_{\Delta_1(w_\infty)}|f'_{\infty}| \leq 1$. In particular, the disk $D_{\infty}$ is ``almost'' vertical. Moreover, by the choice of $q_0$, the set $\mathcal E$ is ``almost'' horizontal. This implies that the disk $D_{\infty}$ is transversal to every branch of $\mathcal E$ and, hence, $\frac{1}{4} \overline{D}_{\infty}$ intersects $\mathcal E$ on a Cantor set. However, since $c_k$ is a compact set for every $k$, then, up to extracting a subsequence, $c_k$ converges in the Hausdorff metric to a subset of some connected component $c_{\infty}$ of the closure of $\pi_w(\frac{1}{4} D_{\infty} \cap \mathcal E)$ and then, by (\ref{inf-eq}), ${\rm diam}(c_{\infty}) \geq \alpha_{\infty}$. This is a contradiction.\end{proof}

For our next argument we need to define some notions. We will call a continuous curve $(z(t), w(t)) : [0, 1] \to {\mathbb C}^2_{z, w}$ a {\em lifting} to ${\mathbb C}^2_{z, w}$ of the curve $z(t) : [0, 1] \to {\mathbb C}_z$ (without restrictions of generality we can assume here that, up to a reparametrisation, if necessary, all the curves are parametrised by the segment $[0, 1]$). Let now $z(t) : [0, 1] \to {\mathbb C}_z$ be a closed (i.e. $z(0) = z(1)$) continuous curve. For a compact set $F$ in ${\mathbb C}^2_{z, w}$ which projects to the given curve $z(t)$ we consider the family $\{\gamma^F_{\alpha}(t)\}_{\alpha \in {\mathcal A}}$ of all liftings $\gamma^F_{\alpha}(t) = (z(t), w^F_{\alpha}(t))$ of the curve $z(t)$ which are contained in the set $F$ (i.e. such that $\gamma^F_{\alpha}(t) \in F$ for all $t \in [0 ,1]$). Then we define the $\it shift$ $\it error$ $\theta(F)$ of the set $F$ as 

\[\theta(F) := \inf_{\alpha \in {\mathcal A}}|w^F_{\alpha}(1) - w^F_{\alpha}(0)|.\]

\noindent
Observe that for two sets $F_1 \subset F_2$ which project to the same curve $z(t)$ one obviously has that

\begin{equation}\label{inf-eq2}
\theta(F_2) \leq \theta(F_1).
\end{equation}

Now we can finally make precise the construction of the Wermer type set $\mathcal E$, specifying conditions on the sequence $\{\varepsilon_n\}_{n \in \mathbb N}$. We first set $\varepsilon_1 = 1$ and set $E_1:=\{(z,w) \in \mathbb C^2:\ w = \varepsilon_1\sqrt{z-a_1}\}$. Then we will choose $\varepsilon_2$ as follows. Fix $0 < r_2 < 1/2$ such that the set $E_1 \cap (\Delta_{r_2}(a_2) \times \mathbb C)$ is the union of the graphs of holomorphic functions $f^1_1, f^1_2: \Delta_{r_2}(a_2) \rightarrow \mathbb C$ and such that, moreover, one has
$$
\kappa_2:=\inf\{|f^1_1(z)-f^1_2(z)|:\ |z-a_2| = r_2\} > 0.
$$
Now, let us choose $\varepsilon_2$ such that $2\varepsilon_2 \sqrt{r_2} < \kappa_2/2$. Then for each $(z,w) \in E_2:=\{(z,w) \in \mathbb C^2:\ w=\sum_{j=1}^2 \varepsilon_j\sqrt{z-a_j}\}$ with $|z-a_2| = r_2$ we consider $w' \in \mathbb C$ such that $(a_2 + (z-a_2)e^{2i\pi}, w') \in E_2$ and observe that
$$
|w-w'| = 2\varepsilon_2 \sqrt{r_2}.
$$
Here we denote by $(a_2 + (z-a_2)e^{2i\pi}, w') \in E_2$ a point which is obtained from $(z,w)$ after one turn around $a_2$ starting at $z$ and keeping it, during this turn, on the set $E_2$.

We can continue the process inductively. Assume that for some $k \geq 3$, $r_2,\dots, r_{k-1}$ and $\varepsilon_1,\dots,\varepsilon_{k-1}$ are already constructed.
We choose $0 < r_k < 1/2$ such that
$$
E_{k-1} \cap \left(\Delta_{r_k}(a_k) \times \mathbb C\right) = \cup_{j=1}^{2^{k-1}}\{(z,f^{k-1}_j(z)):\ z \in \Delta_{r_k}(a_k)\},
$$
where $f^{k-1}_1,\dots,f^{k-1}_{2^{k-1}}$ are functions holomorphic in a neighbourhood of $\overline{\Delta_{r_k}(a_k)}$ and such that for every $1 \leq j\neq l \leq 2^{k-1}$ and every $z \in \partial \Delta_{r_k}(a_k)$, $f^{k-1}_j(z) \neq f^{k-1}_l(z)$.
Then we set
\begin{equation}\label{ka-eq}
\kappa_k:= \inf_{1 \leq j \neq l \leq 2^{k-1}}\{|f^{k-1}_j(z)-f^{k-1}_l(z)|:\ |z-a_k| = r_k\} > 0.
\end{equation}
Let $\varepsilon_k > 0$ be such that
\begin{equation}\label{k-eq}
2 \varepsilon_k \sqrt{r_k} < \frac{\kappa_k}{2}
\end{equation}
and, for every $2 \leq p \leq k-1$,
\begin{equation}\label{p-eq}
2 \varepsilon_k\sqrt{|a_k-a_p|+r_{p}} < \frac{\varepsilon_p \sqrt{r_p}}{2^{k - p + 1}}.
\end{equation}
Condition~(\ref{p-eq}) will insure that the set $E_k$, over the circle $\partial \Delta_{r_p}(a_p)$, will be a sufficiently small perturbation of the set $E_p$ and, hence, the shift error for liftings of the closed circle $\partial \Delta_{r_p}(a_p)$ to a sufficiently small neighbourhood of the set $E_k$ will be bounded from below by a positive constant independent of $k$. Let us make this argument more precise. We first introduce the following notations: for each $k \geq p$ we denote by $E_{k,p}$ the set
$$
E_{k,p} := E_k \cap (\{|z-a_p|=r_p\} \times {\mathbb C}_w),
$$
for each $p \geq 1$ we denote by ${\mathcal E}_p$ the set
$$
{\mathcal E}_p := {\mathcal E} \cap (\{|z-a_p|=r_p\} \times {\mathbb C}_w),
$$
and for each compact set $F \subset {\mathbb C}^2_{z,w}$ and each $\delta > 0$ we denote by $F^\delta$ the set
$$
F^{\delta}:=\cup_{(z,w) \in F}\left(\{z\} \times \Delta_{\delta}(w)\right).
$$
Since, by Condition~(\ref{p-eq}), we know that for each $k > p$ one has
$$
|\varepsilon_k \sqrt{z-a_k}| \leq \varepsilon_k\sqrt{|a_k-a_p|+r_{p}} <  \frac{\varepsilon_p \sqrt{r_p}}{2^{k - p + 2}}
$$
for $z \in \partial \Delta_{r_p}(a_p)$, we conclude that
$$
E_{k,p} \subset E^\frac{\varepsilon_p \sqrt{r_p}}{2^{k - p + 2}}_{k-1,p} \subset E^{{\varepsilon_p \sqrt{r_p}}(\frac{1}{2^{k - p + 2}} + {\frac{1}{2^{k - p + 1}})}}_{k-2,p} \subset E^{{\varepsilon_p \sqrt{r_p}}(\frac{1}{2^{k - p + 2}} + {\frac{1}{2^{k - p + 1}} + \cdots + {\frac{1}{8}})}}_{p,p} \subset E^\frac{\varepsilon_p \sqrt{r_p}}{4}_{p,p}.
$$
Then, after passing to the limit as $k \to \infty$, we will get that
$$
{\mathcal E}_p \subset {\rm cl}\Big({E}^\frac{\varepsilon_p \sqrt{r_p}}{4}_{p,p} \Big),
$$
where for avoiding ambiguity we use the notation ${\rm cl}(X)$ for the closure of $X$. Hence, we also have that
$$
{\rm cl}\Big({\mathcal E}^\frac{\varepsilon_p \sqrt{r_p}}{4}_p \Big)   \subset {\rm cl}\Big({E}^\frac{\varepsilon_p \sqrt{r_p}}{2}_{p,p} \Big).
$$
It follows from Conditions (\ref{ka-eq}) and (\ref{k-eq}) that the shift error of the set $E_{p,p}$ is equal to $2 \varepsilon_p \sqrt{r_p}$. Hence, the shift error of the set $E^\frac{\varepsilon_p \sqrt{r_p}}{2}_{p,p}$ satisfies
$$
\theta \Big({\rm cl}\Big(E^\frac{\varepsilon_p \sqrt{r_p}}{2}_{p,p}\Big) \Big) \geq {\varepsilon_p \sqrt{r_p}}.
$$

By the construction of the Wermer type set $\mathcal E$, it finally follows from the last inclusion and from Property (\ref{inf-eq2}) that for the shift error of the set $\mathcal E$ we have the following
 
\vspace{2mm}
{\bf Property $(\mathcal P)$:} {\sl For every $p \geq 1$, the inequality
	$$
	\theta \Big({\rm cl}\Big({\mathcal E}^\frac{\varepsilon_p \sqrt{r_p}}{4}_p \Big) \Big) \geq {\varepsilon_p \sqrt{r_p}} > 0
	$$
holds.}

\vspace{2mm}
We can now specify the choice of the sequence $\{c(n)\}_n$ and then, using Lemma~\ref{conv-lem}, construct the function $\rho$.

\vspace{2mm}
For every $n \in \mathbb N$, let
$$
\tilde{S}_n:=\left\{z \in \mathbb C:\ -n-2  \leq {\rm Re}(z), {\rm Im}(z) \leq n + 2\right\} \backslash \left\{z \in \mathbb C:\ -n+1 < {\rm Re}(z), {\rm Im}(z) < n - 1\right\}.
$$
and let
$$
\kappa(n):=\inf_{\{p : a_p \in \tilde{S}_n\}}\left\{\frac{\varepsilon_p \sqrt{r_p}}{4}\right\}.
$$

\noindent
The definition of $\tilde{S}_n$ insures that every disc of radius one contained in $\mathbb C$ will be contained in some $\tilde{S}_n$. This property will be used in Section~\ref{main-sect} to prove the Kobayashi hyperbolicity of $\Omega_{\Psi(\rho)}$.

\vspace{2mm}
Now we choose $q(n)$ so large that
\begin{equation}\label{inclusion1-eq}
\{(z,w) \in \mathbb C^2:\ e^{\phi(z,w) + q(n)} < 1 \} \cap (\tilde{S}_n \times \mathbb C) \subset \mathcal E^{\kappa(n)} \cap(\tilde{S}_n \times \mathbb C),
\end{equation}
the function $\phi$ being introduced at the beginning of Subsection~\ref{wermer-subsection}.

\vspace{2mm}
Then, in view of Lemma~\ref{vert-lem}, we can define for every $n > q_0$
 $$
 \delta(n):=\inf\{\delta > 0:\ \beta_k^{\delta} < \frac{1}{2}, \ {\rm for} \ {\rm all} \ n-1 \leq k \leq n + 2\}.
 $$

\noindent 
It follows now that for every $n > q_0$ and every $n-1 \leq k \leq n + 2$ the inequality
 \begin{equation}\label{beta-eq}
 \beta_k^{\delta(n)} \leq \frac{1}{2}
 \end{equation}
 holds and then we choose $\tilde{q}(n) \geq q(n)$ such that
 \begin{equation}\label{inclusion2-eq}
\{(z,w) \in \mathbb C^2:\ e^{\phi(z,w) + \tilde{q}(n)} < 1 \} \cap (\tilde{S}_n \times \mathbb C) \subset \mathcal E^{\delta(n)} \cap (\tilde{S}_n \times \mathbb C).
\end{equation}

Hence, setting $c(n) = q(n) + n$ for $n < q_0$ and $c(n) = \tilde{q}(n) + n$ for $n \geq q_0$, and applying then Lemma~\ref{conv-lem}, we obtain a strictly convex function $\rho:[0,+\infty) \rightarrow [0,+\infty)$ such that for each $n \in \mathbb N$
\begin{equation}\label{conv-eq}
\rho_{|[n-1, n+2]} \geq c(n).
\end{equation}

\noindent
Then the corresponding function $\Psi(\rho)$ defined by (\ref{function-eq}) is plurisubharmonic on $\mathbb C^2$ and the domain $\Omega_{\Psi(\rho)}:=\{(z,w,\zeta) \in \mathbb C^3: {\rm Re}(\zeta) > \Psi(\rho)(z,w)\}$ is a rigid pseudoconvex domain in $\mathbb C^3$.
In Section~\ref{main-sect}, we prove that $\Omega_{\Psi(\rho)}$ satisfies the conditions of Theorem~\ref{main-thm}.

\section{Proofs of Theorem~\ref{main-thm} and Corollary~\ref{main-cor}}\label{main-sect}

\subsection{Proof of Theorem~\ref{main-thm}} We prove that the domain $\Omega_{\Psi(\rho)}$ satisfies all the conclusions of Theorem~\ref{main-thm}. According to Lemma~\ref{core-lem}, the core of $\Omega_{\Psi(\rho)}$ is not empty. Hence, it remains to prove that $\Omega_{\Psi(\rho)}$ is Kobayashi hyperbolic. Assume, to get a contradiction, that there is a point $p \in \Omega_{\Psi(\rho)}$, a neighbourhood $U$ of $p$ in $\mathbb C^3$, $U$ relatively compact in $\Omega_{\Psi(\rho)}$, and for every $k \in \mathbb N$, a holomorphic function $f_k=(z_k,w_k,\zeta_k): \Delta_k(0) \rightarrow \Omega_{\Psi(\rho)}$ such that $f_k(0) \in U$ and $\norm{f_k'(0)}^2:=|z_k'(0)|^2 + |w_k'(0)|^2 + |\zeta_k'(0)|^2 = 1$. Since $\Psi(\rho)$ is nonnegative, it follows that the function ${\rm Re}(\zeta_k)$ is positive on $\Delta_k(0)$ i.e., $\zeta_k(\Delta_k(0)) \subset \mathbb H:=\{\lambda \in \mathbb C: {\rm Re}(\lambda) > 0\}$. For $v \in \mathbb C$ and $\zeta \in \mathbb H$ (resp. $\eta \in \Delta_k(0)$), let $\norm{v}_{\zeta,\mathbb H}$ (respectively $\norm{v}_{\eta,\Delta_k(0)}$) denote the hyperbolic norm of $v$ at $\zeta \in \mathbb H$ (respectively at $\eta \in \Delta_k(0)$). From the decreasing property of the hyperbolic metric (under the action of holomorphic maps) we get
$$
\frac{|\zeta_k'(0)|}{{\rm Re}(\zeta_k(0))} = \norm{\zeta_k'(0) \cdot {\bf 1}}_{\zeta_k(0),\mathbb H} \leq \norm{{\bf 1}}_{0,\Delta_k(0)} = \frac{1}{k}.
$$
Hence, for every $k \geq 1$
\begin{equation}\label{harn-eq}
|\zeta_k'(0)| \leq \frac{|\zeta_k(0)|}{k}.
\end{equation}
Since $U$ is relatively compact in $\Omega_{\Psi(\rho)}$, the set $\{|\zeta_k(0)|,\ k \in \mathbb N\}$ is bounded in $\mathbb C$ and, from (\ref{harn-eq}), there exists $k_0 \geq 0$ such that $|z_k'(0)|^2 + |w_k'(0)|^2 > \frac{1}{2}$ for every $k \geq k_0$. If we set $r_k:=\sqrt{|z_k'(0)|^2 + |w_k'(0)|^2}$, then the holomorphic map $g_k=(\tilde{z}_k,\tilde{w}_k,\tilde{\zeta}_k): \lambda \in \Delta_{kr_k}(0) \mapsto f_k(\lambda/r_k)$ satisfies
$$
|\tilde{z}_k'(0)|^2 + |\tilde{w}_k'(0)|^2 = 1.
$$
It follows now from Lemma~\ref{loc-lem} that, setting $d:=\sup_{(z,w,\zeta) \in U}{\rm Re}(\zeta)$, we will have the following inclusion

\begin{equation}\label{inclu-eq}
\left\{(\tilde{z}_k(\lambda),\tilde{w}_k(\lambda)),\ \lambda \in \Delta_{k/3}\sqrt{2}(0)\right\} \subset F_d:=\left\{(z,w) \in \mathbb C^2:\ \Psi(\rho)(z,w) < 2d\right\}.
\end{equation}
Notice that Condition (\ref{inclu-eq}) is satisfied for every $\lambda \in \Delta_{kr_k/3}$ and hence, since $r_k > 1/\sqrt{2}$, for every $\lambda \in \Delta_{k/3\sqrt{2}}$.
 
\vspace{2mm}
We distinguish two cases.

\vspace{2mm}
\noindent{{\bf Case 1.} \sl There exists an increasing sequence $\{k_m\}_{m \in \mathbb N}$ diverging to $+\infty$, such that for every $m \in \mathbb N$, $|\tilde{z}_{k_m}'(0)| > |\tilde{w}_{k_m}'(0)|$.}

\vspace{2mm}
We will now need the following classical result.

\vspace{1mm}
\noindent{\bf The Bloch Theorem.} {\sl There exists $0 < \theta < 1$ such that for every $r>0$ and every holomorphic function $f: \Delta_r(0) \rightarrow \mathbb C$ with $|f'(0)| = 1$ there are $b \in \mathbb C$ and a holomorphic function $g:\Delta_{{\theta}r}(b) \rightarrow \mathbb C$ such that $f \circ g (\lambda) = \lambda$ for every $\lambda \in \Delta_{{\theta}r}(b)$.}

\vspace{1mm}
Since for every $m \in \mathbb N$ we have $\frac{1}{\sqrt{2}} < |\tilde{z}_{k_m}'(0)| \leq 1$, the function $\lambda \in \Delta_{k_m/6}(0) \mapsto\tilde{z}_{k_m}(\lambda/|\tilde{z}'_{k_m}(0)|)$ satisfies the assumptions of the Bloch Theorem and takes values in $\tilde{z}_{k_m}(\Delta_{k_m/3\sqrt{2}}(0))$. It follows that there exists $0 < \theta < 1$ such that for every $m \in \mathbb N$, there are $b_{k_m} \in \mathbb C$ and a holomorphic function $g_{k_m}:\Delta_{{\theta} k_m/6}(b_{k_m}) \rightarrow \mathbb C$ whose graph $\Gamma(g_{k_m}):=\{(z,g_{k_m}(z)):\ z \in \Delta_{{\theta} k_m/6}(b_{k_m})\}$ will satisfy the condition
\begin{equation}\label{incl2-eq}
\Gamma(g_{k_m}) \subset F_d.
\end{equation}
Now, since $\Psi(\rho)(z,w):=e^{\phi(z,w) + \rho(|Re(z)|) + \rho(|Im(z)|)}$, it follows from Condition~(\ref{inclusion1-eq}) and from the definition of $c(n)$ in Condition~(\ref{conv-eq}) that for every positive integer $n$ such that $\frac{2d}{e^n} < 1$ we have

$$
\begin{array}{lll}
F_d \cap \left(\tilde{S}_n \times {\mathbb C}\right) & = &  \left\{(z,w) \in \mathbb C^2:\ e^{\phi(z,w) + \rho(|x|) + \rho(|y|)} < 2d \right\} \cap \left(\tilde{S}_n \times {\mathbb C}\right)\\
 & & \\
 & \subset & \left\{(z,w) \in \mathbb C^2:\ e^{\phi(z,w) + q(n)} < \frac{2d}{e^n} \right\} \cap \left(\tilde{S}_n \times {\mathbb C}\right) \\
 & & \\
 & \subset & \left\{(z,w) \in \mathbb C^2:\ e^{\phi(z,w) + q(n)} < 1 \right\} \cap \left(\tilde{S}_n \times {\mathbb C}\right) \\
 & & \\
 & \subset & \mathcal E^{\kappa(n)} \cap \left(\tilde{S}_n \times {\mathbb C}\right),
 \end{array}
$$
the last inclusion coming from Condition (\ref{inclusion1-eq}).

\vspace{2mm}
In particular, for every $m \in \mathbb N$,
\begin{equation}\label{contain-eq}
\Gamma(g_{k_m}) \subset \cup_{n \geq n_0} \left(\mathcal E^{\kappa(n)} \cap \left(\tilde{S}_n \times {\mathbb C}\right) \right) \cup K,
\end{equation}
where $n_0$ satisfies $\frac{2d}{e^{n_0}} < 1$ and $K:=F_d \cap\left(\{|z| \leq n_0\} \times \mathbb C\right)$.

\vspace{2mm}
It follows now from the definition of $\kappa(n)$ and Property ($\mathcal P$) that the set
$$
\cup_{n \geq n_0} \left(\mathcal E^{\kappa(n)} \cap \left(\tilde{S}_n \times {\mathbb C}\right) \right)
$$
cannot contain large disks. Here by ``large'' we mean a holomorphic disk of the form $\{(z,w) \in {\mathbb C}^2: w = g(z), z \in \mathcal D\}$ such that its domain of definition $\mathcal D$ contains the disk $\Delta_{\frac{1}{2}}(a_k)$ for some $k \geq n_0$. This contradicts~(\ref{contain-eq}), since  for sufficiently large $m\,\,$ the set $\Gamma(g_{k_m}) \backslash K$ will obviously contain an arbitrarily large disk.

\vspace{2mm}
\noindent{{\bf Case 2.} \sl There exists $k_0 \geq 1$ such that for every $k \geq k_0$, $|\tilde{z}_{k}'(0)| \leq |\tilde{w}_{k}'(0)|$.}

\vspace{1mm}
In particular, we have $|\tilde{w}_{k}'(0)| \geq \frac{1}{\sqrt{2}}$ for every $k \geq k_0$ and, as in the Case 1 above, according to the Bloch Theorem, there exists $0 < \theta < 1$ such that for every $k \geq k_0$, there are $b'_k \in \mathbb C$ and holomorphic functions $h_k:\Delta_{{\theta} k/6}(b'_k) \rightarrow \mathbb C$ whose graph $\Gamma(h_k):=\{(h_k(\lambda),\lambda):\ \lambda \in \Delta_{{\theta} k/6}(b'_k)\}$ satisfies
\begin{equation}\label{incl3-eq}
\Gamma(h_k) \subset F_d.
\end{equation}
 There are two subcases to consider.
 
 \vspace{2mm}
 \noindent{{\bf Subcase 2a.} \sl There is an increasing sequence $(k_m)_{m \in \mathbb N}$ diverging to $+\infty$ and, for every $m \in \mathbb N$, a point $\lambda_{k_m} \in \Delta_{{\theta} k_m/12}(b'_{k_m})$ such that $|h'_{k_m}(\lambda_{k_m})| \geq 1$.}
 
 \vspace{1mm}
 In this case we can repeat the argument of Case 1, replacing $\tilde{z}_{k_m}$ with $h_{k_m}$, $\Delta_{k_m/3\sqrt{2}}(0)$ with $\Delta_{{\theta} k_m/12}(\lambda_{k_m})$ (here we use a trivial observation that for $\lambda_{k_m} \in \Delta_{{\theta} k_m/12}(b'_{k_m})$ one has $\Delta_{{\theta} k_m/12}(\lambda_{k_m}) \subset \Delta_{{\theta} k_m/6}(b'_{k_m})$) and then, using (\ref{incl3-eq}), we obtain the same contradiction as in Case 1.
 
 \vspace{2mm}
 \noindent{{\bf Subcase 2b.} \sl For every $k \in \mathbb N$ large enough and every $\lambda \in \Delta_{{\theta} k/12}(b'_k)$ the inequality $|h'_k(\lambda)| < 1$ holds.}
 
 \vspace{1mm}
 It follows then from Condition~(\ref{conv-eq}), from the definition of $c(n)$ and from Condition~(\ref{inclusion2-eq}) that for every $n \geq q_0$ such that $\frac{2d}{e^n} < 1$ one has
\begin{equation}\label{ineq-eq}
F_d \cap \left(\tilde{S}_n\times \mathbb C\right) \subset  \mathcal E^{\delta(n)} \cap \left(\tilde{S}_n\times \mathbb C\right).
\end{equation} 
 Hence, for every $n \geq q_0$  there exists a compact set $K_n \subset \mathbb C^2$ such that
 $$
 F_d \cap \left(\left\{z \in \mathbb C:\ -n-\frac{1}{2} < {\rm Re}(z) < n +\frac{1}{2},\  -n-\frac{1}{2} < {\rm Im}(z) < n + \frac{1}{2}\right\} \times \mathbb C\right) \subset K_n.
 $$
 
 (See Figure~\ref{Fig4}.)

 \begin{figure}[h!]


\begin{center}
\begin{tikzpicture}[scale=1.2]





\draw[->] [dashed] (0,-3) -- (0,3) ;
\draw[->] [dashed] (-3.7,0) -- (3.7,0) ;



\draw[domain=2:4,samples=100,color=black] plot ({\x},{sqrt(\x/3)+sqrt(\x-2)/6});
\draw[domain=2:4,samples=100,color=black, dotted] plot ({\x},{sqrt(\x/3)+sqrt(\x-2)/6+0.1});

\draw[ultra thick][domain=2.8:3.85,samples=100,color=black] plot ({\x},{sqrt(\x/3)+sqrt(\x-2)/6+(\x-3.1)^3/4});]
\draw (3.3,2.7) node[scale=1] {$\Gamma(h_{k(n)})$};
\draw[->, dashed] (3.3,2.5) -- (3.3,{sqrt(3.3/3)+sqrt(3-2)/6+(3.3-3.1)^3/3+0.03});

\draw[domain=2.4:4,samples=100,color=black, dotted] plot ({\x},{sqrt(\x/3)+sqrt(\x-2)/6-0.1});

\draw (3.9,{sqrt(3.9/3)+sqrt(3.9-2)/6-0.1}) -- (3.9,{sqrt(3.9/3)+sqrt(3.9-2)/6});

\draw[dashed] (4.3,1.3) -- (3.9,{(sqrt(3.9/3)+sqrt(3.9-2)/6-0.1+sqrt(3.9/3)+sqrt(3.9-2)/6)/2});
\draw (4.55,1.3) node[scale=0.5] {$\delta(n)$};

\draw[domain=2:4,samples=100,color=black, dotted] plot ({\x},{sqrt(\x/3)+sqrt(\x-2)/6-0.2});

\draw[domain=2:4,samples=100,color=black] plot ({\x},{sqrt(\x/3)-sqrt(\x-2)/6});
\draw[domain=2:4,samples=100,color=black, dotted] plot ({\x},{sqrt(\x/3)-sqrt(\x-2)/6-0.1});
\draw[domain=2.4:4,samples=100,color=black, dotted] plot ({\x},{sqrt(\x/3)-sqrt(\x-2)/6+0.1});


\draw[domain=0:2,samples=100,color=black] plot ({\x},{sqrt(\x/3)+sqrt(-\x+2)/6});
\draw[domain=0:2,samples=100,color=black, dotted] plot ({\x},{sqrt(\x/3)+sqrt(-\x+2)/6+0.1});

\draw[domain=0:2,samples=100,color=black] plot ({\x},{sqrt(\x/3)-sqrt(-\x+2)/6});
\draw[domain=0.3:2,samples=100,color=black, dotted] plot ({\x},{sqrt(\x/3)-sqrt(-\x+2)/6-0.1});

\draw[domain=0:0.1, samples=100, color=black, dotted] plot ({\x+0.1},{sqrt(\x/3)+sqrt(-\x+2)/6-0.03});
\draw[domain=0:0.07, samples=100, color=black, dotted] plot ({\x+0.1},{-sqrt(\x/3)+sqrt(-\x+2)/6+0.03});

\draw[domain=0.15:1.7, samples=100, color=black, dotted] plot ({\x},{sqrt(\x/3)+sqrt(-\x+2)/6-0.1});
\draw[domain=0.15:1.7, samples=100, color=black, dotted] plot ({\x},{sqrt(\x/3)-sqrt(-\x+2)/6+0.1});

\draw[domain=0:0.07, samples=100, color=black, dotted] plot ({\x+0.1},{sqrt(\x/3)-sqrt(-\x+2)/6-0.03});
\draw[domain=0:0.1, samples=100, color=black, dotted] plot ({\x+0.1},{-sqrt(\x/3)-sqrt(-\x+2)/6+0.03});

\draw[domain=0.15:1.7, samples=100, color=black, dotted] plot ({\x},{-sqrt(\x/3)+sqrt(-\x+2)/6-0.1});
\draw[domain=0.15:1.7, samples=100, color=black, dotted] plot ({\x},{-sqrt(\x/3)-sqrt(-\x+2)/6+0.1});


\draw[domain=-4:0,samples=100,color=black] plot ({\x},{sqrt(-\x/3)+sqrt(-\x+2)/6});
\draw[domain=-4:0,samples=100,color=black, dotted] plot ({\x},{sqrt(-\x/3)+sqrt(-\x+2)/6+0.1});
\draw[domain=-4:-0.2,samples=100,color=black, dotted] plot ({\x},{sqrt(-\x/3)+sqrt(-\x+2)/6-0.1});

\draw[domain=-0.1:0, samples=100, color=black, dotted] plot ({\x-0.1},{sqrt(-\x/3)+sqrt(-\x+2)/6-0.03});
\draw[domain=-0.1:0, samples=100, color=black, dotted] plot ({\x-0.1},{-sqrt(-\x/3)+sqrt(-\x+2)/6+0.03});

\draw[domain=-0.1:0, samples=100, color=black, dotted] plot ({\x-0.1},{sqrt(-\x/3)-sqrt(-\x+2)/6-0.03});
\draw[domain=-0.1:0, samples=100, color=black, dotted] plot ({\x-0.1},{-sqrt(-\x/3)-sqrt(-\x+2)/6+0.03});

\draw[domain=-4:0,samples=100,color=black] plot ({\x},{sqrt(-\x/3)-sqrt(-\x+2)/6});
\draw[domain=-4:-0.2,samples=100,color=black, dotted] plot ({\x},{sqrt(-\x/3)-sqrt(-\x+2)/6+0.1});

\draw[domain=-4:-0.4,samples=100,color=black, dotted] plot ({\x},{sqrt(-\x/3)-sqrt(-\x+2)/6-0.1});


\draw[domain=2:4,samples=100,color=black] plot ({\x},{-sqrt(\x/3)+sqrt(\x-2)/6});
\draw[domain=2:4,samples=100,color=black, dotted] plot ({\x},{-sqrt(\x/3)+sqrt(\x-2)/6+0.1});
\draw[domain=2.4:4,samples=100,color=black, dotted] plot ({\x},{-sqrt(\x/3)+sqrt(\x-2)/6-0.1});

\draw[domain=2:4,samples=100,color=black] plot ({\x},{-sqrt(\x/3)-sqrt(\x-2)/6});
\draw[domain=2:4,samples=100,color=black, dotted] plot ({\x},{-sqrt(\x/3)-sqrt(\x-2)/6-0.1});
\draw[domain=2.4:4,samples=100,color=black, dotted] plot ({\x},{-sqrt(\x/3)-sqrt(\x-2)/6+0.1});

\draw[domain=0:2,samples=100,color=black] plot ({\x},{-sqrt(\x/3)+sqrt(-\x+2)/6});
\draw[domain=0.3:2,samples=100,color=black, dotted] plot ({\x},{-sqrt(\x/3)+sqrt(-\x+2)/6+0.1});

\draw[domain=0:2,samples=100,color=black] plot ({\x},{-sqrt(\x/3)-sqrt(-\x+2)/6});
\draw[domain=0:2,samples=100,color=black, dotted] plot ({\x},{-sqrt(\x/3)-sqrt(-\x+2)/6-0.1});

\draw[domain=-4:0,samples=100,color=black] plot ({\x},{-sqrt(-\x/3)+sqrt(-\x+2)/6});
\draw[domain=-4:-0.4,samples=100,color=black, dotted] plot ({\x},{-sqrt(-\x/3)+sqrt(-\x+2)/6+0.1});

\draw[domain=-4:0,samples=100,color=black] plot ({\x},{-sqrt(-\x/3)-sqrt(-\x+2)/6});
\draw[domain=-4:0,samples=100,color=black, dotted] plot ({\x},{-sqrt(-\x/3)-sqrt(-\x+2)/6-0.1});
\draw[domain=-4:-0.2,samples=100,color=black, dotted] plot ({\x},{-sqrt(-\x/3)-sqrt(-\x+2)/6+0.1});





\draw (0,0) node[scale=0.5] {$\bullet$};
\draw (0.1,-0.15) node[scale=0.5] {$0$};

\draw (0.4,0) node[scale=0.5] {$\bullet$};
\draw (0.5,-0.1) node[scale=0.5] {$1$};
\draw (0.8,0) node[scale=0.5] {$\bullet$};

\draw (-0.4,0) node[scale=0.5] {$\bullet$};
\draw (-0.3,-0.15) node[scale=0.5] {$-1$};
\draw (-0.8,0) node[scale=0.5] {$\bullet$};


\draw (1.9,0) node[scale=0.5] {$\bullet$};
\draw (1.9,-0.15) node[scale=0.5] {$n$};
\draw (1.5,0) node[scale=0.5] {$\bullet$};
\draw (1.5,-0.15) node[scale=0.5] {$n-1$};

\draw (-1.9,0) node[scale=0.5] {$\bullet$};
\draw (-1.95,-0.15) node[scale=0.5] {$-n$};
\draw (-1.5,0) node[scale=0.5] {$\bullet$};
\draw (-1.43,-0.15) node[scale=0.5] {$-n+1$};

\draw (2.3,0) node[scale=0.5] {$\bullet$};
\draw (2.35,-0.15) node[scale=0.5] {$n+1$};

\draw (-2.3,0) node[scale=0.5] {$\bullet$};
\draw (-2.45,-0.15) node[scale=0.5] {$-n-1$};

\draw[dotted, line width=1pt] (-2.1,-1.4) -- (2.1,-1.4) -- (2.1,1.4) -- (-2.1,1.4) -- cycle;

\draw[->][dashed] (-1,2) -- (-1,1.2);
\draw (-1,2.2) node[scale=1] {$K_n$};

\end{tikzpicture}
\end{center}
  \caption{}\label{Fig4}
 \end{figure}

Since $\left\{\Gamma\left(h_k\bigr|_{\Delta_{{\theta} k/12}(b'_k)}\right)\right\}_{k \geq 1}$ forms a sequence of unbounded holomorphic disks, the set $\cup_{k \geq 1}\left(\Gamma\left(h_k\bigr|_{\Delta_{{\theta} k/12}(b'_k)}\right) \cap (\mathbb C^2 \setminus K_n)\right)$ will also contain arbitrarily large discs. Hence, by (\ref{incl3-eq}), (\ref{ineq-eq}) and the definition of $\mathcal E^{\delta(n)}$, the set $\cup_{k \geq 1} \pi_z\left(\Gamma\left(h_k\bigr|_{\Delta_{{\theta} k/12}(b'_k)}\right) \cap (\mathbb C^2 \setminus K_n)\right)$ is not bounded in $\mathbb C_z$.

In particular, we can choose $n \geq q_0$ such that (\ref{ineq-eq}) is satisfied, and $k(n) \geq 1$, $b''_{k(n)} \in \Delta_{({\theta} k(n)/12)-1}(b'_{k(n)})$, such that $h_{k(n)}(b''_{k(n)}) \in T_n$, where $T_n$ is defined in (\ref{tn-eq}). Notice that, according to the assumption of Subcase 2b, the holomorphic disk $\{(h_{k(n)}(\lambda), \lambda); \lambda \in \Delta_1(b''_{k(n)})\}$ belongs to $\mathcal H_n$.
 
 Since $\Delta_1(b''_{k(n)}) \subset \tilde{S}_n$, it follows from (\ref{incl3-eq}) and (\ref{ineq-eq}) that
 $$
 \Gamma\left(h_{k(n)}\bigr|_{\Delta_{1}(b''_{k(n)})}\right) \subset \mathcal E^{\delta(n)} \cap \left(\tilde{S}_n\times \mathbb C\right).
 $$
However, ${\rm diam}\left(\pi_w\left(\Gamma\left(h_{k(n)}\bigr|_{\Delta_{1}(b''_{k(n)})}\right)\right)\right) = 1$, which contradicts Condition~(\ref{beta-eq}).
This completes the proof of Theorem~\ref{main-thm}. \qed

\vskip 0,3cm
\noindent{\bf {Remark 3.}} Observe, that the statement which is actually proved in the Case 1 and Case 2 above using the Bloch Theorem, can be formulated as the following property of the domain $F_d$:

\vspace{2mm}
{\bf Property $(\mathcal F)$:} {\sl For each $d > 0$ there exists $r = r(d) > 0$ such that the domain $F_d$ contains no holomorphic disks of radius $r > r(d)$ $\rm$(the last part of the statement means, more precisely, that for every holomorphic map $h : \Delta_r(0) \to F_d$ such that $\norm{h'(0)} = 1$ one has $r \leq r(d)$$\rm)$.}

\vspace{2mm}
We give here an explicit formulation of this property because it will be needed in the forthcoming paper \cite{SZ19}.

\subsection{Proof of Corollary~\ref{main-cor}} Assume, to get a contradiction, that $\varphi$ is a strong antipeak function at infinity for $\Omega_{\Psi(\rho)}$. Then $\varphi\bigr|_{\mathcal E}$ is continuous psh. and bounded from above in a neighbourhood of $\mathcal E$. It follows now from the same argument as in ~\cite[Theorem~2.2]{HST18} that $\varphi \equiv C$ on ${\mathcal E}$ for some $C \in \mathbb R$. Since, by the definition of a strong antipeak function, we know that $\varphi(z) \to 0$ for $z \in \mathcal E$ as $\norm{z} \rightarrow \infty$, we conclude that $\varphi \equiv 0$ on $\mathcal E$. This contradicts the definition of an antipeak function and, hence, completes the proof of Corollary~\ref{main-cor}. \qed


\end{document}